\documentclass[11pt]{amsart} 
\usepackage{amssymb,amsmath,latexsym,enumerate,graphicx,bbm,mathptmx,microtype,cite}

\newtheorem{theorem}{Theorem} 
\newtheorem{corollary}[theorem]{Corollary}

\newtheorem{lemma}[theorem]{Lemma}

\newtheorem{exam}{Example}
\newenvironment{example}{\begin{exam}\rm}{\end{exam}}
\newtheorem*{rem}{Remarks}

\renewcommand\emptyset{\varnothing}

\newcommand\commentout[1]{}
\newcommand\Def[1]{{\bf #1}}

\newcommand\Asc{\operatorname{Asc}} 
\newcommand\Des{\operatorname{Des}} 
\newcommand\asc{\operatorname{asc}} 
\newcommand\des{\operatorname{des}}

\newcommand\RR{\mathbb{R}}

\begin{document}

\title{Bivariate Order Polynomials}

\author{Matthias Beck}
\address{Department of Mathematics\\
         San Francisco State University\\
         San Francisco, CA 94132\\
         U.S.A.}
\email{mattbeck@sfsu.edu}

\author{Maryam Farahmand}
\address{Department of Mathematics, Mills College\\
Oakland, CA 94613\\
         U.S.A.}
\email{mfarahmand@mills.edu}

\author{Gina Karunaratne}
\address{College of Alameda\\
555 Ralph Appezzato Memorial Parkway\\
Alameda, CA 94501\\
         U.S.A.}
\email{gkarunaratne@peralta.edu}

\author{Sandra Zuniga Ruiz}
\address{Graduate School of Education, University of California-Berkeley \\
 Berkeley, CA 94720\\
         U.S.A. }
\email{zuni5344@berkeley.edu}


\begin{abstract}
Motivated by Dohmen--P\"onitz--Tittmann's bivariate chromatic polynomial $\chi_G(x,y)$, which counts all $x$-colorings of a graph $G$ such that adjacent vertices get different
colors if they are $\le y$, we introduce a bivarate version of Stanley's order polynomial, which counts order preserving maps from a given poset to a chain. Our results include
decomposition formulas in terms of linear extensions, a combinatorial reciprocity theorem, and connections to bivariate chromatic polynomials.
\end{abstract}



\date{8 January 2019}

\maketitle


\section{Introduction}

Graph coloring problems are ubiquitous in many areas within and outside of mathematics.
The motivation of our study is the \Def{bivariate chromatic polynomial} $\chi_G(x,y)$ of a graph $G = (V, E)$, first introduced in \cite{dohmenponitztittmann} and
defined as the counting function of colorings $c : V \to [x] := \left\{ 1, 2, \dots, x \right\}$ that satisfy for any edge $vw \in E$
\[
  c(v) \ne c(w) \qquad \text{ or } \qquad c(v) = c(w) > y \, .
\]
The usual univariate chromatic polynomial of $G$ can be recovered as the special evaluation $\chi_G(x,x)$.
Dohmen, P\"onitz, and Tittmann provided basic properties of $\chi_G(x,y)$ in~\cite{dohmenponitztittmann}, including polynomiality and special evaluations yielding the matching and independence polynomials of $G$.
Subsequent results include a deletion--contraction formula and applications to Fibonacci-sequence identities~\cite{hillarwindfeldt}, common generalizations of $\chi_G(x,y)$ and the Tutte polynomial~\cite{averbouchgodlinmakowsky}, and closed formulas for paths and cycles~\cite{dohmenbivariatepathsandcycles}.

For a finite poset $(P, \preceq)$, Stanley~\cite{stanleychromaticlike} (see also~\cite[Chapter 3]{stanleyec1}) famously introduced the ``chromatic-like'' \Def{order polynomial} $\Omega_P(x)$ counting all order preserving maps $\varphi : P \to [x]$, that is, 
\[
  a \preceq b \qquad \Longrightarrow \qquad \varphi(a) \le \varphi(b) \, .
\]
Here we think of $[x] = \{ 1, 2, \dots, x \}$ as a chain with $x$ elements, and so $\le$ denotes the usual order in $\RR$.
The connection to chromatic polynomials is best exhibited through a variant of $\Omega_P(x)$, namely the number $\Omega_P^\circ(x)$ of all strictly order preserving maps $\varphi : P \to [x]$:
\[
  a \prec b \qquad \Longrightarrow \qquad \varphi(a) < \varphi(b) \, .
\]
When thinking of $P$ as an acyclic directed graph, it is a short step interpreting $\Omega_P^\circ(x)$ as a directed version of the chromatic polynomial.
Along the same lines, one can write the chromatic polynomial of a given graph $G$ as
\begin{equation}\label{eq:chiintoomegas}
  \chi_G(x) \ = \sum_{ \sigma \text{ acyclic orientation of } G } \Omega_\sigma^\circ(x) \, .
\end{equation}
Stanley's two main initial results on order polynomials were
\begin{itemize}
\item decomposition formulas for $\Omega_P(x)$ and $\Omega_P^\circ(x)$ in terms of certain permutation statistics for linear extensions of~$P$, from which polynomiality of
$\Omega_P(x)$ and $\Omega_P^\circ(x)$ also follows;
\item the combinatorial reciprocity theorem
$
  (-1)^{ |P| } \, \Omega_P(-x) \ = \ \Omega_P^\circ(x) \, .
$
\end{itemize}
The latter, combined with~\eqref{eq:chiintoomegas}, gives in turn rise to
\begin{itemize}
\item Stanley's reciprocity theorem for chromatic polynomials: $(-1)^{ |V| } \, \chi_G(-x)$ equals the number of pairs of an $x$-coloring and a compatible acyclic
orientation~\cite{stanleyacyclic}.
\end{itemize}
Our goal is to extend these three results to the bivariate setting.\footnote{
A reciprocity theorem for the bivariate chromatic polynomial was previously given in~\cite{bichromatic}; unfortunately, its statement and proof are wrong.
}
Along the way, we introduce bivariate versions of the order polynomials which, we believe, are interesting in their own right.

We will call a poset $(P, \preceq)$ \Def{bicolored} if we think of $P$ as the disjoint union of $C$ and $S$, which we call the \Def{celeste} and \Def{silver} elements of $P$.
We call $\varphi : P \to [x]$ an \Def{order preserving $(x,y)$-map} if
\[
  a \preceq b \ \Longrightarrow \ \varphi(a) \le \varphi(b) \ \ \text{ for all } a, b \in P
  \qquad \text{ and } \qquad
  \varphi(c) \ge y \ \ \text{ for all } c \in C \, .
\]
Moreover, $\varphi : P \to [x]$ is a \Def{strictly order preserving $(x,y)$-map} if
\[
  a \prec b \ \Longrightarrow \ \varphi(a) < \varphi(b) \ \ \text{ for all } a, b \in P
  \qquad \text{ and } \qquad
  \varphi(c) > y \ \ \text{ for all } c \in C \, .
\]
The functions $\Omega_{ P, \, C } (x, y)$ and $\Omega^{\circ}_{ P, \, C } (x, y)$ count the number of order preserving $(x, y)$-maps and strictly order preserving $(x,y)$-maps, respectively.  We note that these are natural extensions of the single-variable order polynomials, as $\Omega_P (x) = \Omega_{ P, \, C } (x, 1)$ and $\Omega_P^\circ (x) = \Omega_{ P, \, C }^\circ (x, 0)$.
It is \emph{a priori} not clear why $\Omega_{ P, \, C } (x, y)$ and $\Omega^{\circ}_{ P, \, C } (x, y)$ are polynomials in two variables, but the following simple example gives an indication.

\begin{example} Consider the bicolored chain $P = \{a_1, a_2\}$ where $a_1 \prec a_2$ and $a_2$ is celeste. We would like to find all possible strictly order preserving $(x,y)$-maps $\varphi: P \rightarrow [x]$ such that $\varphi(a_2) > y$. The first case
\[
1 \le \varphi(a_1) \le y < \varphi(a_2) \le x
\]
gives $y$ choices for $\varphi(a_1)$ and $x-y$ choices for $\varphi(a_2)$, whereas the second case
\[
1 \le y < \varphi(a_1)  < \varphi(a_2) \le x
\]
yields $\binom{x-y}{2}$ choices. Therefore 
$\Omega^{\circ}_{ P, \, C }(x,y) 
= \frac{1}{2}(x^2-x-y^2-y)$.
\end{example}

We now give an overview of the main results in this paper.
After setting up machinery for bicolored chains in Section~\ref{sec:chains}, we will prove decomposition formulas for $\Omega_{ P, \, C } (x, y)$ and $\Omega^{\circ}_{ P, \, C
} (x, y)$ in terms of linear extensions of~$P$ in Section~\ref{sec:decomp} (see Theorem~\ref{thm:decomptypew} below), as well as:

\begin{theorem}\label{thm:mainposet}
The functions $\Omega_{ P, \, C } (x, y)$ and $\Omega^{\circ}_{ P, \, C } (x, y)$ are polynomials in $x$ and $y$.
They satisfy the reciprocity theorem
\[
  (-1)^{ |P| } \, \Omega^{\circ}_{ P, \, C } (-x,-y) \ = \ \Omega_{ P, \, C } (x,y+1) \, .
\]
\end{theorem}

To state our reciprocity theorem for bivariate chromatic polynomials, which we will prove in Section~\ref{sec:chromatic}, we recall that a \Def{flat} of a graph $G$ is a graph
$H$ that can be constructed from $G$ by a series of contractions. We denote by $V(H)$ the vertex set of $H$, and by $C(H)$ the set of vertices of $H$ that resulted from contractions of~$G$.

\begin{theorem}\label{thm:maingraph}
For any graph $G=(V,E)$,
\[
  \chi_G(-x,-y) \ = \sum_{ H \text{ \rm flat of } G } (-1)^{|V(H)|} \, m_{H}(\sigma, c) \, ,
\]
where $m_H(\sigma, c)$ is the number of pairs $(\sigma, c)$ consisting of an acyclic orientation  $\sigma$ of $H$ and a compatible coloring $c: V(H) \rightarrow [x]$ such that
$(v) > y$ if $v \in C(H)$.
\end{theorem}

Here an orientation and a coloring $c$ are \Def{compatible} if $c(v) \le c(w)$ for anyx edge oriented from $v$ to~$w$.


\section{Bicolored Chains}\label{sec:chains}

\begin{lemma}\label{lem:firstchaincount}
If $P=\{a_1 \prec \cdots \prec a_{k} \prec a_{k+1} \prec \cdots \prec a_n\}$ is a bicolored chain of length $n$ where $a_{k+1}$ is the minimal celeste element, then
\[
\Omega^{\circ}_{ P, \, C } (x, y) \ = \ \sum_{i=0}^{k}\binom{y}{i} \binom{x-y}{n-i}
\qquad \text{ and } \qquad
\Omega_{P}(x, y) \ = \ \sum_{i=0}^{k} \binom{y-2+i}{i} \binom{x-y+n-i}{n-i}.
\]
\end{lemma}

\begin{proof}
We start with the count of strictly order preserving maps $\varphi: P \rightarrow [x]$ with $\varphi(a_{k+1})>y$. This gives us a set of inequalities for $\varphi$ with $k+1$
cases: for
\[
1 \le \varphi(a_1) < \cdots<\varphi(a_{k}) \le y < \varphi(a_{k+1}) < \cdots<  \varphi(a_{n}) \le x
\]
there are $\binom{y}{k}\binom{x-y}{n-k}$ possible maps $\varphi$. The second case
\[
1 \le \varphi(a_1) < \cdots<\varphi(a_{k-1}) \le y < \varphi(a_{k}) < \cdots<  \varphi(a_{n}) \le x
\]
yields $\binom{y}{k-1}\binom{x-y}{n-k+1}$ possible maps $\varphi$. We repeat this process up to the last case
\[
1 \le  y < \varphi(a_{1}) < \cdots<  \varphi(a_{n}) \le x
\]
for which there are $\binom{y}{0}\binom{x-y}{n}$ possible maps $\varphi$.

The count of the weakly order preserving maps $\varphi: P \rightarrow [x]$ with $\varphi(a_{k+1}) \ge y$ follows in a similar fashion.
\end{proof}

We remark that this lemma proves polynomiality of $\Omega^{\circ}_{ P, \, C } (x,y)$ and $\Omega_{ P, \, C } (x, y)$, and we obtain the following precursor of Theorem~\ref{thm:mainposet} for chains.

\begin{corollary}
If $P=\{a_1 \prec \cdots \prec a_{k} \prec a_{k+1} \prec \cdots \prec a_n\}$ is a bicolored chain of length $n$ where $a_{k+1}$ is the minimal celeste element, then
\[
(-1)^n \, \Omega^{\circ}_{ P, \, C } (-x,-y) \ = \ \Omega_{ P, \, C } (x, y+1) \, .
\]
\end{corollary}

\begin{proof}
By Lemma~\ref{lem:firstchaincount} and the elementary identity $(-1)^d \binom{-n}{d}=\binom{n+d-1}{d}$,
\begin{align*}
 \Omega^{\circ}_{ P, \, C } (-x,-y)
 \ &= \ \sum_{i=0}^{k} \binom{-y}{i}\binom{-x+y}{n-i}\\
   &= \ \sum_{i=0}^{k} (-1)^i \binom{y+i-1}{i} (-1)^{n-i}\binom{x-y+n-i-1}{n-i}\\
   &= \ (-1)^n \, \Omega_{ P, \, C } (x, y+1) \, . \qedhere
\end{align*}
\end{proof}

Next we need a refinement of the above results for certain chains that arise from a given general poset on $n$ elements.
A \Def{linear extension} of the bicolored poset $P$ is a chain $L$ whose underlying set (including the celeste/silver partition) equals that of $P$ such that
\[
  a \preceq_P b \qquad \Longrightarrow \qquad a \preceq_L b \, .
\]
Similar to the situation in classical poset theory, linear extensions are best recorded as permutations of a fixed labeling $\omega: P \rightarrow [n]$.
Namely, to a given linear extension $L = \{a_1 \prec a_{2} \prec \dots \prec a_n\}$ of $P$, we associate the word $\omega_L:=\omega(a_1)\omega(a_2)\cdots\omega(a_n)$.
We collect the \Def{ascents} and \Def{descents} of the word $\omega_L$ in the sets
\[
  \Asc(\omega_L) := \left\{ j : \omega(a_j) < \omega(a_{j+1}) \right\} 
  \qquad \text{ and } \qquad
  \Des(\omega_L) := \left\{ j : \omega(a_j) > \omega(a_{j+1}) \right\} 
\]
respectively.
The cardinality of these sets are denoted by $\asc(\omega_L) := |\Asc(\omega_L)|$ and $\des(\omega_L) := |\Des(\omega_L)|$, respectively.

An order preserving $(x,y)$-map $\varphi$ of $L$ is \Def{of type $\omega_L$} if for  $i \in \operatorname{Asc}(\omega_L)$ we have $\varphi(a_i) \le \varphi(a_{i+1})$, and for  $i \in \operatorname{Des}(\omega_L)$ we have $\varphi(a_i) < \varphi(a_{i+1})$.
Let $\Omega_{\omega_L}(x,y)$ denote the number of type-$\omega_L$ order preserving $(x,y)$-maps; note that this implies that $\varphi(c) \ge y$ for the minimal celeste element
$c \in L$. The definition of $\Omega^{\circ}_{\omega_L}(x,y)$ is identical except that now $\varphi(c) > y$.

\begin{example}
Consider the poset $P$ and labeling $\omega$ in Figure \ref{P13}. Its three linear extensions $L_1$, $L_2$, and $L_3$ come with the associated words $\omega_{L_1}=12345$, $\omega_{L_2}=12435$, and $\omega_{L_3}=14235$, respectively.

\begin{figure}[htb]
\begin{center}
  \includegraphics [width=3.5in]{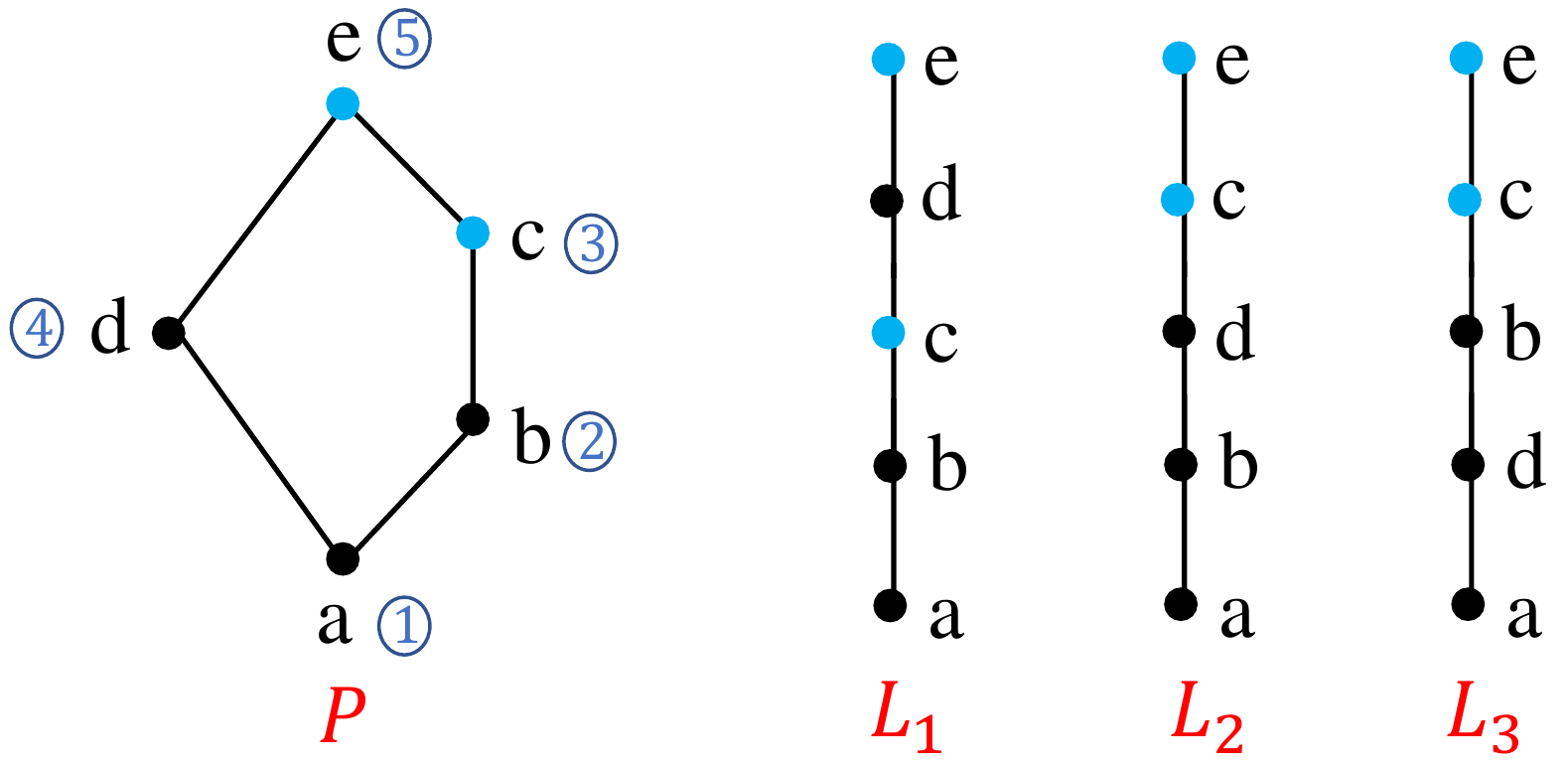}
\end{center}
\caption{A bicolored poset $P$ with labeling $\omega$ and its linear extensions.}\label{P13}
\end{figure}

As an example, we will compute the number of type-$\omega_{L_3}$ order preserving $(x,y)$-maps $\varphi$. 
The word $\omega_{L_3}=14235$ has a descent at $i=2$. Thus, a type-$14235$ order preserving $(x,y)$-map satisfies the set of inequalities
\begin{equation}\label{kldhfgv}
1 \le \varphi(a) \le \varphi(d) < \varphi(b) \le \varphi(c) \le \varphi(e) \le x \, .
\end{equation}
We remove equalities by defining a bijection $\varphi \mapsto \bar{\varphi}$ via
\begin{align*}
\bar{\varphi} (a)&=\varphi(a) \\
  \bar{\varphi} (d)&=\varphi(d)+\operatorname{asc}(14)= \varphi(d)+1 \\
    \bar{\varphi} (b)&=\varphi(b)+\operatorname{asc}(142)= \varphi(b)+1 \\
     \bar{\varphi} (c)&=\varphi(c)+\operatorname{asc}(1423)=\varphi(c)+2 \\
       \bar{\varphi} (e)&=\varphi(e)+\operatorname{asc}(14235)=\varphi(e)+3 \, .
       \end{align*}
Now \eqref{kldhfgv} becomes
\[
1 \le \bar{\varphi}(a) < \bar{\varphi}(d) < \bar{\varphi}(b) < \bar{\varphi}(c) < \bar{\varphi}(e) \le x+3 \, .
\]
The additional inequality $\varphi(c) >y$ gives rise to $\bar{\varphi}(c)>y+2$.
Lemma~\ref{lem:firstchaincount} now yields
\[
\Omega_{14235}^{\circ}(x,y) \ = \ \sum_{i=0}^{2}\binom{y-2}{i}\binom{x-3-y+2}{5-i}.
\]
This example contains essentially all ingredients needed for the proof of our next lemma.
\end{example}

\begin{lemma}\label{lem:typewexpansions}
Let $P$ be a bicolored poset with a fixed labeling $\omega: P \rightarrow [n]$.
Given a linear extension $L = \{a_1 \prec a_{2} \prec \dots \prec a_n\}$ of $P$ with associated word $\omega_L = \omega({a_1})\cdots\omega(a_n)$, let
$a_{k+1}$ be the minimal celeste element in $L$ and denote $\widetilde{\omega}_L = \omega({a_1})\cdots\omega(a_{k+1})$. Then
\[
\Omega_{\omega_L}^{\circ}(x, y) \ = \ \sum_{i=0}^{k}\binom{y-\operatorname{asc} (\widetilde{\omega}_L)}{i} \binom{x-y+\operatorname{asc}(\widetilde{\omega}_L)-\operatorname{asc} (\omega_L)}{n-i}
\]
and
\[
\Omega_{\omega_L}(x, y) \ = \ \sum_{i=0}^{k}\binom{y-\operatorname{des}(\widetilde{\omega}_L)-2+i}{i} \binom{x-y+\operatorname{des}(\widetilde{\omega}_L)-\operatorname{des} (\omega_L)+n-i}{n-i}.
\]
\end{lemma}

\begin{proof} We want to count the number of order preserving $(x,y)$-maps  $\varphi: L \rightarrow [x]$ of type $\omega_L$. The map $\varphi$ satisfies the set of inequalities $\varphi(a_{i}) \le \varphi(a_{i+1})$ if $i$ is an ascent of $\omega_L$, and $\varphi(a_{i}) < \varphi(a_{i+1})$ if $i$ is a descent of $\omega_L$. We eliminate weak inequalities by creating a bijection $\varphi \rightarrow \bar{\varphi}$ defined by $\bar{\varphi}(a_1)  = \varphi(a_1)$ and
\[
\bar{\varphi}(a_{i}) = \varphi(a_{i}) +\operatorname{asc} \left( \omega(a_1) \cdots \omega(a_i) \right) \qquad \text{for }  2 \le i \le n.
\]
If $i \in \operatorname{Des}(\omega_L)$, then 
$\operatorname{asc} \left( \omega(a_1) \cdots \omega(a_{i+1}) \right) - \operatorname{asc} \left( \omega(a_1) \cdots \omega(a_{i}) \right) = 0$, and so
$$\varphi(a_i) < \varphi(a_{i+1})  \qquad \Longrightarrow \qquad \bar{\varphi}(a_i) <\bar{\varphi}(a_{i+1}) \, .$$
If $i \in \operatorname{Asc}(\omega_L)$, then 
$\operatorname{asc} \left( \omega(a_1) \cdots \omega(a_{i+1}) \right) - \operatorname{asc} \left( \omega(a_1) \cdots \omega(a_{i}) \right) = 1$, whence
$$\varphi(a_i) \le \varphi(a_{i+1})  \qquad \Longrightarrow \qquad \bar{\varphi}(a_i) <\bar{\varphi}(a_{i+1}) \, .$$
In addition, $\varphi(a_{k+1}) > y$ implies $\bar{\varphi}(a_{k+1})>y+\operatorname{asc} (\widetilde{\omega}_L)$, and with
\[
1 \le \bar{\varphi}(a_{1}) < \cdots < \bar{\varphi}(a_{n}) \le x+\operatorname{asc} (\omega_L) \, ,
\]
the formula for $\Omega_{\omega_L}^{\circ}(x, y)$ now follows from Lemma~\ref{lem:firstchaincount}.

The formula for $\Omega_{\omega_L}(x, y)$ follows in an analogous fashion, replacing all ascents in our argumentation by descents: now our bijection is defined via $\bar{\varphi}(a_1)  = \varphi(a_1)$ and
\[
\bar{\varphi}(a_{i}) = \varphi(a_{i}) - \des \left( \omega(a_1) \cdots \omega(a_i) \right) \qquad \text{for }  2 \le i \le n,
\]
giving rise to
$\bar{\varphi}(a_{k+1}) \ge y+\operatorname{des} (\widetilde{\omega}_L)$ and 
\[
1 \le \bar{\varphi}(a_{1}) \le \cdots \le \bar{\varphi}(a_{n}) \le x+\operatorname{des} (\omega_L) \, . \qedhere
\]
\end{proof}

The \Def{inverse} of $\omega=\omega_1 \cdots \omega_n$ is the word $\overline{\omega}$ defined by $\overline{\omega}_j := \omega_{n+1-j}$. This switches all ascents with
descents; in particular, $\operatorname{asc}(\omega_L)= \operatorname{des}(\overline{\omega}_L)$ and $\operatorname{des}(\omega_L)= \operatorname{asc}(\overline{\omega}_L)$.

\begin{corollary}\label{ghki}
Let $P$ be a bicolored poset with a fixed labeling $\omega: P \rightarrow [n]$,
and let $L = \{a_1 \prec a_{2} \prec \dots \prec a_n\}$ be a linear extension of $P$, with associated word $\omega_L = \omega({a_1})\cdots\omega(a_n)$.
Then
$$(-1)^n \, \Omega^{\circ}_{\omega_L}(-x,-y) \ = \ \Omega_{\overline{\omega}_L}\, (x,y+1).$$
\end{corollary}

\begin{proof}
Let $a_{k+1}$ be the minimal celeste element in $L$ and let $\widetilde{\omega}_L = \omega (a_1) \cdots \omega(a_{k+1})$. By Lemma~\ref{lem:typewexpansions} and the fact that $\operatorname{asc}(\widetilde{\omega}_L)= \operatorname{des}(\overline{\widetilde{\omega}}_L)$, 
\begin{align*}
\Omega^{\circ}_{\omega_L}(-x,-y)
  \ &= \ \sum_{i=0}^{k}\binom{-y+\operatorname{asc} (\widetilde{\omega}_L)}{i}\binom{-x+y+\operatorname{asc} (\widetilde{\omega}_L)-\operatorname{asc}(\omega_L)}{n-i} \\
    &= \ \sum_{i=0}^{k} (-1)^i \binom{y-\operatorname{asc} (\widetilde{\omega}_L)+i-1}{i} (-1)^{ n-i } \binom{x-y-\operatorname{asc} (\widetilde{\omega}_L)+\operatorname{asc}(\omega_L)+n-i-1}{n-i}\\
    &= \ (-1)^n \sum_{i=0}^{k}  \binom{y-\operatorname{des} (\overline{\widetilde{\omega}}_L)+i-1}{i} \binom{x-y-1-\operatorname{des} (\overline{\widetilde{\omega}}_L)+\operatorname{des}(\overline{\omega}_L)+n-i}{n-i}\\
    &= \ (-1)^n \, \Omega_{\overline{\omega}_L}\, (x,y+1) \, . \qedhere
\end{align*}
\end{proof}


\section{Bicolored Posets}\label{sec:decomp}

To apply Lemma~\ref{lem:typewexpansions} to general posets, we recall the notions of a \Def{natural labeling} $\omega: P \rightarrow [n]$ of a poset $P$ with $n$ elements, that
is,
\[
  a \prec b \qquad \Longrightarrow \qquad \omega(a) < \omega(b) \, ,
\]
and a \Def{reverse natural labeling}:
\[
  a \prec b \qquad \Longrightarrow \qquad \omega(a) > \omega(b) \, .
\]

\begin{theorem}\label{thm:decomptypew} Let $P$ be a bicolored poset with a reverse natural labeling $\omega$ and a natural labeling $\tau$. Then
\[
\Omega^{\circ}_{ P, \, C }(x,y) \ = \ \sum_L \, \Omega^{\circ}_{\omega_L}\, (x,y)
\qquad \text{ and } \qquad
\Omega_{ P, \, C }(x,y) \ = \ \sum_L \, \Omega_{\tau_L}\, (x,y)
\]
where both sums are over all possible linear extensions $L$ of~$P$.
\end{theorem}

\begin{proof}
We will show that there is a bijection between the set of strictly order preserving $(x,y)$-maps $\varphi$ and the set of all pairs $(L, \hat{\varphi})$ where $L$ is a linear extension and $\hat{\varphi}$ is a type-$\omega_L$ map on the bicolored poset $P$.

    Suppose $\varphi$ is a strictly order preserving $(x,y)$-map. We will build a  linear extension $L$ for which $\varphi$ is of type $\omega_L$.
\begin{itemize}
\item If $\varphi(a_i) < \varphi(a_j)$, then let $a_i \prec a_j$ in $L$. Note that $\omega (a_i) > \omega(a_j)$ because $\omega$ is a natural reverse labeling.
\item If $\varphi(a_i) = \varphi(a_j)$ with $\omega(a_i) < \omega(a_j)$, then let $a_i \prec a_j$ in $L$. 
\item If $\varphi(a_i) = \varphi(a_j)$ with $\omega(a_i) > \omega(a_j)$, then let $a_i \succ a_j$ in $L$. 
 \end{itemize}
This gives us the linear extension $L$ with associated word $\omega_L$, and by construction, $\varphi$ is of type~$\omega_L$.
 
Conversely, let $L$ be a linear extension of $P$ and let $\hat{\varphi}$ be a type-$\omega_L$ map. Type-$\omega_L$ maps agree with strictly order preserving maps between comparable elements in a poset by definition. The map $\hat{\varphi}$ may give us weak or strict inequalities between incomparable elements in $P$. Strictly order preserving maps place no restrictions on what happens between incomparable elements in the poset. Thus, $\hat{\varphi}$ is also a strictly order preserving map on~$P$.

The proof of the second part is  similar to that of the first part with one major difference: now we have $\tau(a_i) < \tau(a_j)$ when $a_i \prec a_j$ in the linear extension
$L$ because $\tau$ is a natural labeling.
\end{proof}

\begin{proof}[Proof of Theorem~\ref{thm:mainposet}]
Polynomiality follows from Lemma~\ref{lem:typewexpansions} and Theorem~\ref{thm:decomptypew}.

To prove the reciprocity statement, fix a natural reverse labeling $\omega$ of $P$; note that $\overline{\omega}$ is a natural labeling. By Corollary~\ref{ghki} and Theorem~\ref{thm:decomptypew},
\[
  \Omega^{\circ}_P(-x,-y)
  \ = \ \sum_{L} \Omega^{\circ}_{\omega_L}(-x,-y)
  \ = \ (-1)^n \sum_{L} \Omega_{\overline{\omega}_L}(x,y+1)
  \ = \ (-1)^n \, \Omega_P (x,y+1) \, . \qedhere
\]
\end{proof}


\section{Bivariate Chromatic Polynomials}\label{sec:chromatic}
 
In preparation for our proof of Theorem~\ref{thm:maingraph}, we show the following decomposition result.

\begin{lemma}\label{lem:graphdecomp}
\[
  \chi_G(x,y) \ = \sum_{ H \text{ \rm flat of } G } \sum_{ \substack{ \sigma \text{ \rm acyclic} \\ \text{\rm orientation of } H } } \!\!\!\!\Omega^{\circ}_{\sigma, \,C(H)} (x,y) \, .  \]
\end{lemma}

\begin{proof}
Let $c : V \to [x] := \left\{ 1, 2, \dots, x \right\}$ be a coloring of $G$ that satisfies
\[
  c(v) \ne c(w) \qquad \text{ or } \qquad c(v) = c(w) > y \, .
\]
Let $H$ be the flat of $G$ obtained by contracting all edges whose endpoints have the same color; note that the colors of $C(H)$ will be $> y$.
We orient the edges of $H$ by their color gradient: $v \to u$ if $c(v) < c(u)$.
The resulting orientation will be acyclic, and the resulting colorings of $H$ will precisely be counted by~$\Omega^{\circ}_{\sigma, \, C(H)} (x,y)$.

This argumentation can be reversed: given a flat $H$ of $G$ and an acyclic orientation $\sigma$ of $H$, a strictly order preserving $(x,y)$-map counted by
$\Omega^{\circ}_{\sigma, \, C(H)} (x,y)$ can be extended to a coloring of $G$ where vertices that result in contractions for $H$ get colors~$> y$.
\end{proof}

\begin{example}
Our proof of Lemma~\ref{lem:graphdecomp} is illustrated for the case $G = K_3$ in Figures~\ref{fig:contractions} and~\ref{fig:acyclic}.
\begin{figure}[htb]
\begin{center}
  \includegraphics [width=2.5in]{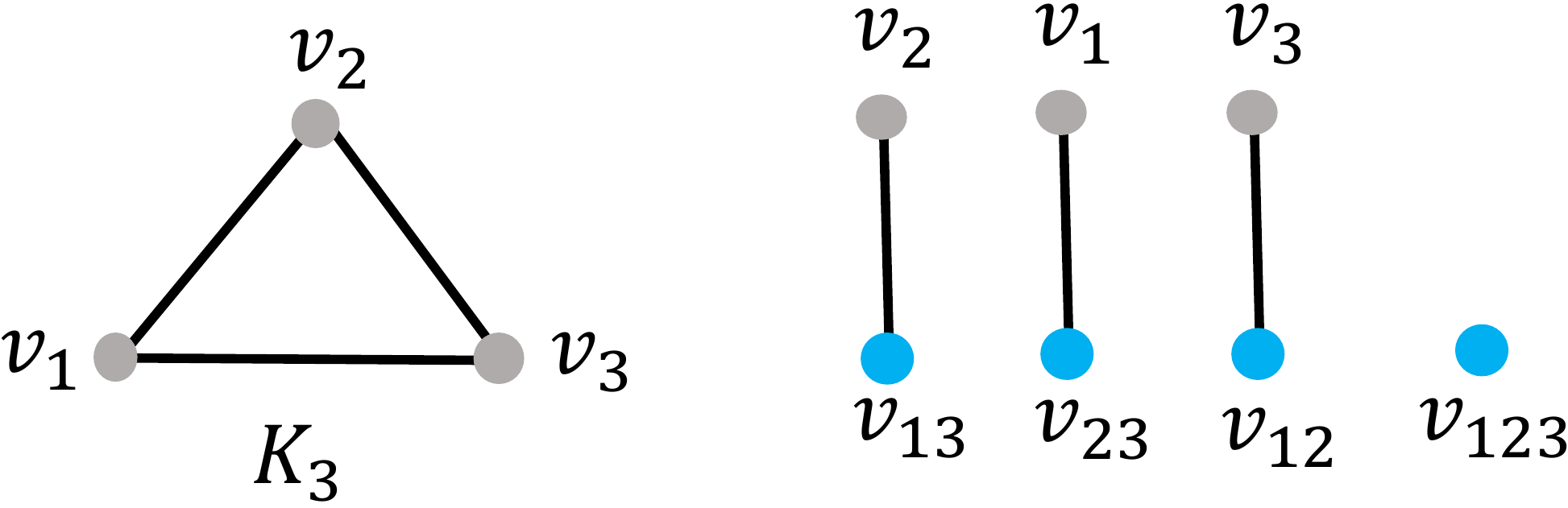}
\end{center}
\caption{Contractions of $K_3$.}\label{fig:contractions}
\end{figure}
\begin{figure}[htb]
\begin{center}
  \includegraphics [width=3in]{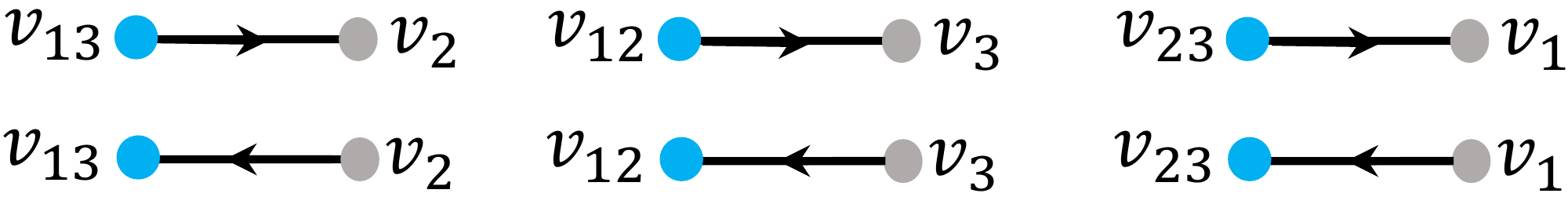}
\end{center}
\caption{Acyclic orientations of contractions of $K_3$.}\label{fig:acyclic}
\end{figure}

For $H = G$ we obtain
$\Omega^{\circ}_{ \sigma, \, \emptyset}(x,y) = \binom x 3$ for any acyclic orientation $\sigma$, and summing over all six such orientations gives the univariate chromatic
polynomial of~$K_3$.

If $H$ is the result of one contraction, we need to consider two orientations $\sigma_1$ and $\sigma_2$ with order polynomials
\[
  \Omega^{\circ}_{\sigma_1,v_{12}}(x,y)=\binom{x-y}{2}+(x-y)y
  \qquad \text{ and } \qquad
  \Omega^{\circ}_{\sigma_2,v_{12}}(x,y)=\binom{x-y}{2} \, .
\]
Finally, if $H$ results from contracting two edges, $\Omega^{\circ}_{ \sigma, \, v_{ 123 } }(x,y) = x-y$.
Thus
\[
  \chi_{K_3}(x,y)
  \ = \ 6 \binom x 3 + 3 \left( \binom{x-y}{2}+(x-y)y + \binom{x-y}{2} \right) + x-y
  \ = \ x^3-3xy+y \, .
\]
\end{example}
 
\begin{proof}[Proof of Theorem~\ref{thm:maingraph}]
By Theorem~\ref{thm:mainposet} and Lemma~\ref{lem:graphdecomp},
\[
  \chi_G(-x,-y) \ = \sum_{ H \text{ \rm flat of } G } \sum_{ \substack{ \sigma \text{ \rm acyclic} \\ \text{\rm orientation of } H } } \!\!\!\!(-1)^{|V(H)|} \, \Omega_{\sigma,\, C(H)} (x,y+1) \, .
\]        
Here $\Omega_{\sigma,\, C(H)} (x,y+1)$ counts the number of order preserving maps $\varphi: \sigma \rightarrow [x]$ subject to two conditions:
First, if $c \in C(H)$ then $\varphi(c) \ge y+1$; and second, $\varphi$ is compatible with $\sigma$.
\end{proof}
 

\bibliographystyle{amsplain}
\bibliography{bib}

\setlength{\parskip}{0cm} 

\end{document}